\documentclass[11pt,reqno,a4paper]{amsart}

\usepackage{amsmath,amsfonts,amsthm,mathrsfs,amssymb}
\usepackage[usenames]{color}

\usepackage{enumerate}
\usepackage{hyperref}
\usepackage{xcolor}
\hypersetup{
  colorlinks,
  linkcolor={blue!80!black},
  urlcolor={blue!80!black},
  citecolor={blue!80!black}
}

\numberwithin{equation}{section}
\usepackage[capitalize,noabbrev]{cleveref}

\usepackage{doi}
\usepackage[sort,numbers]{natbib}

\theoremstyle{plain}
\newtheorem{theorem}{Theorem}[section]
\newtheorem*{theorem*}{Theorem}
\newtheorem{lemma}[theorem]{Lemma}

\newtheorem{proposition}[theorem]{Proposition}
\newtheorem*{proposition*}{Proposition}

\newtheorem*{conjecture*}{Conjecture}

\theoremstyle{definition}

\theoremstyle{remark}

\newtheorem*{remark*}{Remark}

\newcommand{\abs}[1]{\left\lvert #1 \right\rvert}
\newcommand{\norm}[1]{\left\lVert #1 \right\rVert}

\newcommand{\R}{\mathbb{R}}
\newcommand{\C}{\mathbb{C}}
\newcommand{\Z}{\mathbb{Z}}
\newcommand{\N}{\mathbb{N}}

\title[Nonradiating sources vanish at corners]{Nonradiating sources
  and transmission eigenfunctions vanish at corners and edges}

\author{Eemeli Bl{\aa}sten}
\address{Jockey Club Institute for Advanced Study, Hong Kong University of Science and Technology, Hong Kong SAR.}
\email{iaseemeli@ust.hk}

\begin{document}

\begin{abstract}
  We consider the inverse source problem of a fixed wavenumber: study
  properties of an acoustic source based on a single far- or
  near-field measurement. We show that nonradiating sources having a
  convex or non-convex corner or edge on their boundary must vanish
  there. The same holds true for smooth enough transmission
  eigenfunctions. The proof is based on an energy identity from the
  enclosure method and the construction of a new type of planar
  complex geometrical optics solution whose logarithm is a branch of
  the square root. The latter allows us to deal with non-convex
  corners and edges.

  \medskip 
  \noindent{\bf Keywords} inverse source problem, nonradiating, corner
  scattering, complex geometrical optics, interior transmission
  eigenfunction.

  \medskip
  \noindent{\bf Mathematics Subject Classification (2010)}: 35P25,
  78A46 (primary); 51M20, 81V80 (secondary).
\end{abstract}

\maketitle

\section{Introduction and results}
The inverse source problem is a longstanding open problem in
scattering theory. More formally let $f$ have compact support, $f =
\chi_\Omega \varphi$, $\Omega\subset\R^n$ a bounded domain and
$\varphi\in L^\infty(\R^n)$. Given a wavenumber $k>0$ the source
produces a scattered wave $u\in H^2_{loc}(\R^n)$ given by
\[
(\Delta+k^2) u = f, \qquad \lim_{r\to\infty} r^\frac{n-1}{2}
\big(\partial_r - ik \big) u = 0
\]
where $r=\abs{x}$. Then $u$ can be expanded as follows
\[
u(x) = \frac{e^{ik\abs{x}}}{\abs{x}^{(n-1)/2}} u^\infty(\hat{x}) +
\mathcal O(\abs{x}^{n/2})
\]
where $\hat x = x/\abs{x}$ is the radial variable. The function
$u^\infty\in L^2(\mathbb S^{n-1})$ is called the \emph{far-field
  pattern} of $u$ and models the measurements. Knowing $u^\infty$
cannot determine $f$. The reason is simple, if not informal:
$u^\infty$ is a function of $n-1$ independent variables, and $f$ is a
function $n$ independent variables.

An alternative formulation for the inverse source problem is the
following: given the Fourier transform $\hat f(\xi)$ for
$\abs{\xi}=k>0$ fixed, what can one determine from $f$ or its support?
This question is of great interest both for harmonic analysis and
applications in scattering theory. Studying the properties of
so-called \emph{nonradiating sources} has picked up interest in the
recent years. For the mathematical and physical background we refer to
the excellent thesis \cite{Gbur}. In short the inverse source problem
is ill-posed and its solution requires a-priori knowledge. In contrast
to the inverse scattering problem where there is some control on the
incident waves, the inverse source problem has received less attention
because. Despite this there are several reasons coming from
applications for understanding the source problem. For instance if the
object of interest is far away or otherwise not accessible to incident
waves, then one can only study its radiated field.

In this paper we consider convex polyhedral sources and scatterers
with only one measurement. An early paper, but for the conductivity
equation, is \cite{Friedman--Isakov} where the unique determination of
constant penetrable inclusions was proved by one boundary
measurement. This was then extended by many
\cite{Alessandrini--Isakov, Barcelo--Fabes--Seo, Kang--Seo1,
  Kang--Seo2, Seo}. A reconstruction formula for the convex hull of a
penetrable inclusion of the conductivity appeared in
\cite{IkehataConductivity}, and in \cite{IkehataScattering} a nonzero
wavenumber is considered. In \cite{Ikehata} the convex hull is
reconstructed for the inverse scattering and source
problems. Impenetrable obstacles have also been reconstructed, see
\cite{IkehataObstacle} and the references therein. For more details
see Section~2.2 of \cite{IkehataSurvey} and also
\cref{IkehataComparison} in this article.

Our research started with an alternative point of view. Theorems for
the interior transmission problem imply that given an acoustic
potential, at certain wavenumbers there is a sequence of normalized
incident waves that produce arbitrarily small far-field patterns.  In
\cite{BPS} it was shown that acoustic potentials with corners always
scatter any incident wave and so the far-field energy cannot reach
zero. As a consequence transmission eigenfunctions must vanish at
corners \cite{BLvanishing, BLvanishingNum}. Other studies \cite{HSV,
  PSV, EH} generalise the geometric setting and show uniqueness for
shape determination by one incident wave--far-field pattern pair as
long as the scatterer is polyhedral. In \cite{BLpiecewise} the whole
acoustic potential is determined if it is known to be piecewise
constant on a given grid. A recent result of similar flavour
determines a potential from a finite-dimensional subspace by a finite
number of incident-wave--far-field pattern pairs or boundary Cauchy
data \cite{Alberti--Santacesaria}.

\medskip
The first theorem in this article is for the inverse source
problem. Any source that has a corner or edge singularity
scatters. This result is surprising in the sense that there are
nonradiating sources that are piecewise constant in one dimension
\cite{Gbur}. Similarly, $f = f_0 \chi_{B(0,r_0)}$ is nonradiating of
wavenumber $k$ in three dimensions if $kr_0$ is a zero of the first
order spherical Bessel function. Somehow corners radiate more than
finite curvature. The exact formulation is in \cref{thm1}.
\begin{theorem}
  Let $f=\chi_\Omega\varphi$ with $\varphi$ H\"older-continuous around
  a convex or non-convex corner or edge point $x_c\in\partial\Omega$
  that can be connected to infinity. If $(\Delta+k^2)u=f$ where $u$
  satisfies the Sommerfeld radiation condition but $u^\infty=0$ then
  $\varphi(x_c)=0$.
\end{theorem}

\medskip
A topic deeply tied to inverse scattering and nonradiation is the
interior transmission problem. The wavenumber $k$ is a transmission
eigenvalue for an acoustic potential if the latter does not scatter
some incident wave. The converse is more delicate \cite{BPS}. We refer
to the survey \cite{CH}. The transmission eigenvalue problem for a
potential $V\in L^\infty$ on a bounded domain $\Omega\subset\R^n$ asks
for $k>0$ and $v,w\in L^2(\Omega)$ such that
\begin{equation} \label{TE}
  \begin{cases}
    (\Delta+k^2)w=0 & \qquad\Omega, \\
    (\Delta+k^2(1+V))v=0 & \qquad\Omega, \\
    v-w \in H^2_0(\Omega), \quad \norm{w}_{L^2(\Omega)}=1. &
  \end{cases}
\end{equation}
If this has a solution then $k$ is called a \emph{transmission
  eigenvalue}, and the pair $(v,w)$ \emph{transmission
  eigenfunctions}.

The first result about intrinsic properties of the eigenfunctions
themselves that we know of was shown in
\cite{BLvanishing,BLvanishingNum} and the addendum
\cite{BLvanishingAdd}. These showed that transmission eigenfunctions
vanish at corners if they are $H^2$-smooth nearby.  Using the new
methods presented here for the inverse source problem we can
generalise the geometric setting and the a-priori assumptions on the
eigenfunctions and potentials in \cite{BLvanishing}. The full
statement and proof is in \cref{vanishingThm}.
\begin{theorem}
  Let $(v,w)\in L^2(\Omega)\times L^2(\Omega)$ be transmission
  eigenfunctions for the wavenumber $k>0$ and acoustic potential
  $V$. If $x_c\in\partial\Omega$ is a corner point or edge around
  which $V$ is H\"older continuous and $v,w$ are $H^2$-smooth, then
  $v(x_c)=w(x_c)=0$.
\end{theorem}

\medskip
Our final result is on shape determination. A single far-field
measurement gives a function of $n-1$ independent variables, so maybe
the shape of the source could be recovered?  This is known as
Schiffer's problem \cite{CK}: does $u^\infty$ determines $\Omega$ for
a fixed $k$? The exact conditions under which this happens is an open
problem. If $f=0$ then $u^\infty=0$. However there are
\emph{nonradiating sources}. For example consider $v \in
H^2_0(\Omega)$ and let $f = (\Delta+k^2)v$. Then $u=v$ is the wave
produced by $f$ and $u^\infty=0$. What is recoverable in general is
the \emph{scattering support}: see \cite{Kusiak--Sylvester,
  SylvesterIP} and the references therein. Despite the lack of
uniqueness for general sources, there are numerical algorithms,
e.g. \cite{InverseSourceNUM1, InverseSourceNUM2} for recovering
sources that can be approximated by well-separated point-sources. An
even earlier result by \cite{Ikehata} recovers the convex hull of a
polygonal source from a single measurement. The corner scattering
methods used to prove \cref{thm1,vanishingThm} can also show the
theorem below.
\begin{theorem}
  Let $f=\chi_\Omega\varphi$, $f'=\chi_{\Omega'}\varphi'$ with
  $\Omega,\Omega'$ convex polyhedra and $\varphi,\varphi'$ bounded
  functions that are H\"older continuous near the corners and possibly
  edges of $\Omega,\Omega'$, respectively. If $(\Delta+k^2)u=f$,
  $(\Delta+k^2)u'=f'$ and $u^\infty={u'}^\infty$, then
  $\Omega=\Omega'$ and $\varphi=\varphi'$ on the corners and the edges
  where they are H\"older-continuous.
\end{theorem}
As it turns out the two-dimensional case of this result is equivalent
to what was shown in \cite{Ikehata}. Its exact formulation is
\cref{thm2}. More comparison is done in \cref{IkehataComparison}.

\medskip
In contrast with the fairly involved series expansion method for
corner scattering by \cite{EH} we introduce a new type of
\emph{complex geometrical optics solution} and use it in energy- or
orthogonality identities. The original CGO solutions \cite{Calderon,
  Sylvester--Uhlmann} were used for solving the Calder\'on problem, or
the inverse potential scattering problem for a fixed wavenumber $k$
with infinitely many incident-wave--far-field pairs. Their principal
part is an exponential of a complex linear function. The
two-dimensional case, at the time still unsolved, required the use of
a new type of solutions where the logarithm of the principal part is a
second degree polynomial \cite{Bukhgeim}. These are all so-called
\emph{limiting Carleman weights} all of which were classified in
\cite{LCW1,LCW2}.

In the single measurement setting, the enclosure method \cite{Ikehata}
and the original corner scattering argument and following papers
\cite{BPS, HSV, PSV, BLstability, BLpiecewise, BLvanishing} always
used the linear phase form of the CGO solutions which decay in a
half-space. A limitation of that argument is that only convex corners
could be considered unlike in the fairly involved analysis of
\cite{EH}. The Bukhgeim solutions of \cite{Bukhgeim} would make the
situation even worse in two dimensions: only corners smaller than
$\pi/2$ could be dealt with. Here we remove this limitation. By
choosing a suitable branch of $\sqrt{x_1+ix_2}$ we note that
\[
u_0 = \exp(\sqrt{s}\sqrt{x_1+ix_2}), \qquad s\in\R_+
\]
is harmonic and decays exponentially in all directions except for the
lone branch cut. This is useful in higher dimensions too by a
dimension reduction argument. Methods from \cite{Bukhgeim, BTW} likely
extend this function to a solution
\[
u_0 = \exp(\sqrt{s}\sqrt{x_1+ix_2})(1+r(x))
\]
of the equation with potential, $(\Delta+q)u_0=0$, but we only need
the harmonic version in this paper, so skip the full Neumann series
construction.

\section{A harmonic function decaying almost everywhere}
In this section we construct a new type of test function that's
harmonic and decays exponentially in all directions except one. An
integral identity, \cref{u0intByParts}, allows us to use it for corner
scattering instead of the complex geometrical optics solutions. It is
already known that more general solutions can be used in inverse
problems, e.g. by having more general \emph{limiting Carleman weights}
than the linear $x\mapsto\rho\cdot x$, see e.g. \cite{LCW1,LCW2}, and
also a complex analytic phase function in \cite{Bukhgeim}. By having
an exponential that decays in all directions except one we can
consider any angle in corner scattering, including non-convex
ones. The construction is done by taking a non-principal branch of
$\exp((x_1+ix_2)^a)$ where $0<a<1$. Let us study its properties.

\begin{lemma}\label{isAnalytic}
  Let $a\in\R$ be a given exponent, $z=x_1+ix_2\in\C$ the variable,
  and $m\in\Z$ a branch number. Define the complex exponential $\exp
  z^a$ by
  \[
  u_0^{a,m}(z) := \exp\left( \abs{z}^a \big(\cos a(\arg z + 2\pi m) +
  i \sin a(\arg z + 2\pi m)\big) \right).
  \]
  Then it is complex analytic in $\C\setminus(\R_-\cup\{0\})$ if the
  complex argument is defined as $-\pi < \arg z \leq \pi$.
\end{lemma}
\begin{proof}
  The function $z\mapsto z^a = \exp(a \log_m(z))$ is analytic in the
  domain of analyticity of $\log_m(z)$, the $m$'th branch of the
  complex logarithm, which is given by
  \[
  \log_m(z) = \ln\abs{z} + i(\arg z + 2\pi m).
  \]
  The branch cut of $\arg z$ is $\R_-\cup\{0\}$ so $\log_m(z)$ is
  analytic in $\C\setminus(\R_-\cup\{0\})$. Then
  \[
  u_0^{a,m}(z) = \exp\exp( a \log_m(z) )
  \]
  is complex analytic in that same set.
\end{proof}

\begin{lemma}\label{decaySector}
  Let $a=1/2$ and $m=1$ for $u_0^{a,m}$ from \cref{isAnalytic}.  For
  $x\in\R^2$ denote $r=\abs{x}$, $\theta=\arg(x_1+ix_2)$. We define
  \begin{equation}\label{2Du0Def}
    u_0(x) := u_0^{a,m}(x_1+ix_2) = \exp\big( \sqrt{r} \big( \cos
    (\tfrac{\theta}{2} + \pi) + i \sin (\tfrac{\theta}{2} + \pi)
    \big)\big).
  \end{equation}
  Then $\Delta u_0 = 0$ in $\R^2 \setminus ({\R_-\times\{0\}}
  \cup\{(0,0)\})$, and $s\mapsto u_0(s x)$ decays exponentially in
  $\R_+$ whenever $x$ is in that same domain of harmonicity.
\end{lemma}
\begin{proof}
  The formula and domain of analyticity, which implies harmonicity,
  follow directly from \cref{isAnalytic}. Then
  \[
  \abs{u_0(sx)} = \exp\big( \sqrt{s}\sqrt{r}
  \cos(\tfrac{\theta}{2}+\pi) \big).
  \]
  For $x$ in the domain of harmonicity we have $\tfrac{\theta}{2}+\pi
  \in \tfrac{1}{2}{]{-\pi,\pi}[}+\pi$. This implies that the cosine is
  evaluated in ${]{\tfrac{\pi}{2},\tfrac{3\pi}{2}}[}$ where it is
  negative.
\end{proof}

\begin{proposition}\label{2DupperAndLower}
  Let $u_0:\R^2\to\C$ from \cref{decaySector} and consider the open
  sector
  \[
  \mathcal C = \{x\in\R^2 \mid x\neq0,\, \theta_m < \arg(x_1+ix_2) <
  \theta_M \}
  \]
  for angles $-\pi < \theta_m < \theta_M < \pi$.

  Let $\alpha,s>0$. Then
  \begin{equation}\label{2Dupper}
    \int_{\mathcal C} \abs{u_0(sx)} \abs{x}^\alpha dx \leq
    \frac{2(\theta_M-\theta_m) \Gamma(2\alpha+4)}{\delta_{\mathcal
        C}^{2\alpha+4}} s^{-\alpha-2}
  \end{equation}
  where $\delta_{\mathcal C} = -\max_{\theta_m<\theta<\theta_M}
  \cos(\theta/2+\pi) > 0$. Moreover
  \begin{equation}\label{2Dlower}
    \int_{\mathcal C} u_0(sx) dx = 6i ( e^{{-2}{\theta_M}i} -
    e^{{-2}{\theta_m}i} ) s^{-2}
  \end{equation}
  and for $h>0$
  \begin{equation}\label{2DfarFromCorner}
    \int_{\mathcal C \setminus B(0,h)} \abs{u_0(sx)} dx \leq
    \frac{6(\theta_M-\theta_m)}{\delta_{\mathcal C}^4}
    s^{-2} e^{-\delta_{\mathcal C} \sqrt{h s}/2}.
  \end{equation}
\end{proposition}
\begin{proof}
  By \cref{decaySector} and $dx = r dr d\theta$
  \begin{align*}
    &\int_{\mathcal C} \abs{u_0(sx)} \abs{x}^\alpha dx =
    \int_{\theta_m}^{\theta_M} \int_0^\infty \exp \big( \sqrt{sr}
    \cos(\tfrac{\theta}{2}+\pi) \big) r^{\alpha+1} dr d\theta
    \\ &\qquad \leq \int_{\theta_m}^{\theta_M} d\theta \int_0^\infty
    \exp({-\delta_{\mathcal C}} \sqrt{sr}) r^{\alpha+1} dr
  \end{align*}
  where $\delta_{\mathcal C} > 0$. Change variables $t =
  \delta_{\mathcal C} \sqrt{sr}$ so $r = (t/\delta_{\mathcal C})^2
  s^{-1}$ and $dr = 2t dt/(\delta_{\mathcal C}^2 s)$. Recall the
  definition $\Gamma(\beta) = \int_0^\infty \exp({-t}) t^{\beta-1} dt$
  and the first claim follows. In the third claim we have instead
  $\alpha=0$ and the incomplete gamma function $\Gamma(\beta,d) =
  \int_d^\infty t^{\beta-1} \exp(-t) dt$ which we bound by
  \[
  \Gamma(\beta,d) \leq e^{-d/2} \int_d^\infty t^{\beta-1} e^{-t/2} dt
  = e^{-d/2} 2^\beta \int_{d/2}^\infty s^{\beta-1} e^{-s} ds \leq
  2^\beta \Gamma(\beta) e^{-d/2}
  \]
  and this implies the estimate.

  For the second claim we start similarly,
  \[
  \int_{\mathcal C} u_0(sx) dx = \int_{\theta_m}^{\theta_M}
  \int_0^\infty \exp\big( \sqrt{sr} \big( \cos(\tfrac{\theta}{2}+\pi)
  + i \sin(\tfrac{\theta}{2}+\pi) \big) \big) r dr d\theta,
  \]
  but the change of variables is more complicated and done as a path
  integration in the complex plane. Denote $\omega =
  -\cos(\theta/2+\pi)-i\sin(\theta/2+\pi)$ and note $\Re\omega>0$ when
  $\theta_m<\theta<\theta_M$. Then
  \[
  \int_0^\infty \exp\big({-\omega}\sqrt{sr}\big) r dr = \frac{2}{s^2}
  \int_0^\infty \exp\big({-\omega}t\big) t^3 dt.
  \]
  Since $\Re\omega>0$ the integral converges and we have
  \[
  \int_0^\infty \exp\big({-\omega}t\big) t^3 dt =
  \lim_{\varepsilon\to0+} \int_\varepsilon^{1/\varepsilon} \exp
  \big({-\omega}t\big) t^3 dt = \lim_{\varepsilon\to0+}
  \omega^{-4} \int_{A_\varepsilon} f(z) dz
  \]
  where $A_\varepsilon = \{ z \in \C \mid z = \abs{z} \omega,\,
  \varepsilon < \abs{z} < 1/\varepsilon \}$ and $f(z) = \exp({-z})
  z^3$ is an entire function. We shall deal with the case of
  $\arg\omega>0$. The case $\arg\omega<0$ is dealt with similarly, and
  $\arg\omega=0$ is what we'll reduce the other cases to. Consider the
  contour integral over the path defined by
  \begin{align*}
    &A_\varepsilon = \{z\in\C \mid
    \varepsilon<\abs{z}<1/\varepsilon,\, \arg z = \arg\omega\},
    \\ &B_\varepsilon = \{z\in\C \mid \abs{z}=1/\varepsilon,\, 0<\arg
    z<\arg\omega\},\\ &C_\varepsilon = \{z\in\C \mid
    \varepsilon<\abs{z}<1/\varepsilon,\, \arg z=0\}, \\ &D_\varepsilon
    = \{z\in\C \mid \abs{z}=\varepsilon,\, 0<\arg z<\arg\omega\}
  \end{align*}
  in the anti-clockwise direction. The boundedness of $f$ implies that
  the integral over $D_\varepsilon$ vanishes as $\varepsilon\to0$. On
  $B_\varepsilon$, the function's modulus has maximum
  $\exp(-\varepsilon^{-1}\cos\arg\omega) \varepsilon^{-2}$ and the
  path has length $\varepsilon^{-1}\arg\omega$. Their product tends to
  zero as $\varepsilon\to0$ because $\Re\omega>0$ so
  $\cos\arg\omega>0$ and the exponential wins. Hence
  \[
  \int_{A_\varepsilon} f(z) dz = - \int_{C_\varepsilon} f(z) dz =
  \int_\varepsilon^{1/\varepsilon} e^{-t} t^3 dt \to \Gamma(4) = 6.
  \]
  by the holomorphicity of $f$. Combining all the integrals above we
  arrive at
  \[
  \int_0^\infty \exp\big({-\omega}\sqrt{sr}\big) r dr =
  \frac{12}{\omega^4 s^2}.
  \]
  Let us take the integral over $\theta$ next. Use the more convenient
  notation $\omega = -\exp i(\theta/2+\pi)$ so $\omega^4 = \exp
  i(2\theta+4\pi) = \exp i2\theta$. Thus
  \[
  \int_{\theta_m}^{\theta_M} \int_0^\infty \exp\big( {-\omega}
  \sqrt{sr}\big) r dr d\theta = \frac{12}{s^2}
  \int_{\theta_m}^{\theta_M} \exp( {-i}2\theta) d\theta
  \]
  and the claim follows.
\end{proof}

\section{Source corner scattering}
\begin{lemma} \label{sourceIdentity}
  Let $\Omega\subset\R^n$ be a bounded Lipschitz domain, $k\geq0$ and
  $u,u'\in H^2(\Omega)$, $f,f'\in L^2(\Omega)$ be functions such
  that
  \[
  (\Delta+k^2)u=f, \qquad (\Delta+k^2)u'=f'.
  \]
  Given any $u_0 \in H^2(\Omega)$ satisfying $(\Delta+k^2)u_0=0$, we
  have
  \[
  \int_\Omega (f-f')u_0 dx = \int_{\partial\Omega} \big( u_0
  \partial_\nu (u-u') - (u-u') \partial_\nu u_0 \big) d\sigma.
  \]
\end{lemma}
\begin{proof}
  The equations imply that
  \begin{align*}
    &\int_\Omega (f-f')u_0 dx = \int_\Omega u_0 (\Delta+k^2)(u-u') dx
    \\ &\qquad = \int_{\partial\Omega} \big( u_0 \partial_\nu (u-u') -
    (u-u') \partial_\nu u_0 \big) d\sigma
  \end{align*}
  because $(\Delta+k^2)u_0=0$.
\end{proof}

We cannot use \cref{sourceIdentity} directly because $u_0 \not\in
H^2$ near the origin. Instead we have to pass by a limit and use the
fact that in applications we will have $u=u'$ on the boundary
integral.
\begin{lemma} \label{u0intByParts}
  Let $u_0:\R^2\to\C$ from \cref{decaySector} and
  \[
  \mathcal C = \{x\in\R^2 \mid x\neq0,\,
  \theta_m<\arg(x_1+ix_2)<\theta_M \}
  \]
  for given angles $-\pi<\theta_m<\theta_M<\pi$. Assume that $u,u'\in
  H^2(\mathcal C\cap B)$ where $B=B(0,h)$ for some $h>0$. Moreover let
  $u=u'$ and $\partial_\nu u = \partial_\nu u'$ in $B\cap
  \partial\mathcal C$. Then
  \begin{align*}
    &\int_{\mathcal C\cap B} u_0(sx) \Delta(u-u') dx \\& \qquad =
    \int_{\mathcal C\cap\partial B} \big( u_0(sx) \partial_\nu (u-u')
    - (u-u') \partial_\nu (u_0(sx)) \big) d\sigma
  \end{align*}
  for $s>0$.
\end{lemma}
\begin{proof}
  Let $\Omega_\varepsilon = (\mathcal C\cap B) \setminus
  B(0,\varepsilon)$ for $0<\varepsilon<h$. Since $\abs{u_0} \leq 1$ we
  have
  \[
  \int_{\mathcal C\cap B} u_0 \Delta(u-u') dx = \lim_{\varepsilon\to0}
  \int_{\Omega_{\varepsilon}} u_0 \Delta(u-u') dx
  \]
  and recall also that $u_0$ is smooth and harmonic in that same
  set. Hence Green's formula, and the condition of $u=u'$,
  $\partial_\nu u=\partial_\nu u'$ on $\partial\mathcal C\cap B$ give
  \begin{align*}
    &\ldots = \int_{\mathcal C\cap\partial B} \big( u_0 \partial_\nu
    (u-u') - (u-u') \partial_\nu u_0 \big) d\sigma \\ &\qquad
    -\lim_{\varepsilon\to0} \int_{\mathcal C\cap S(0,\varepsilon)}
    \big( u_0 \partial_\nu (u-u') - (u-u') \partial_\nu u_0 \big)
    d\sigma
  \end{align*}
  where $\partial_\nu = \partial_r$ is the radial derivative.

  We have $\abs{u_0(sx)} = \exp(\sqrt{sr}\cos(\theta/2+\pi))\leq1$ and
  \[
  \abs{\partial_r (u_0(sx))} = \abs{ \frac{\sqrt{s}
      \exp{(\theta/2+\pi)i} }{2\sqrt{r}}
    e^{\sqrt{sr}\exp(\theta/2+\pi)i} } \leq \frac{1}{2}
  \sqrt{\frac{s}{r}}.
  \]
  Note also that $H^2(\mathcal C\cap B) \hookrightarrow
  C^{1/2}(\overline{\mathcal C\cap B})$ in two (and three) dimensions,
  and that $(u-u')(0)=0$. Hence
  \[
  \abs{u(x)-u'(x)} = \abs{(u-u')(x) - (u-u')(0)} \leq C
  \norm{u-u'}_{H^2(\mathcal C\cap B)} \abs{x}^{1/2}
  \]
  for $x\in \mathcal C\cap B$ and in particular for $x\in \mathcal
  C\cap S(0,\varepsilon)$. This shows that
  \begin{align*}
    &\abs{\int_{\mathcal C\cap S(0,\varepsilon)} (u-u')(x)
      \partial_\nu(u_0(sx)) d\sigma(x) } \\ &\qquad \leq C
    \norm{u-u'}_{H^2(\mathcal C\cap B)} \varepsilon^{1/2} \frac{1}{2}
    \sqrt{\frac{s}{\varepsilon}} (\theta_M-\theta_m)\varepsilon \to 0
  \end{align*}
  as $\varepsilon\to0$.

  For the remaining term, we can estimate $\abs{u_0}\leq1$ and use
  Cauchy-Schwartz to get
  \begin{equation}\label{sphereIntegral}
    \abs{\int_{\mathcal C\cap S(0,\varepsilon)} u_0 \partial_\nu
      (u-u') d\sigma(x) } \leq \sqrt{(\theta_M-\theta_m)\varepsilon}
    \norm{ \partial_r (u-u') }_{L^2(\mathcal C\cap S(0,\varepsilon))}.
  \end{equation}
  Denote $g(x) = \partial_r (u-u')(x)$. Then $g\in H^1(\mathcal C\cap
  B)$ and $\norm{g}_{H^1} \leq \norm{u-u'}_{H^2}$ in any given open
  set. Let $G(y) = g(\varepsilon y)$. Then $d\sigma(x) = \varepsilon
  d\sigma(y)$ so
  \[
  \norm{g}_{L^2(\mathcal C\cap S(0,\varepsilon))} = \sqrt{\varepsilon}
  \norm{G}_{L^2(\mathcal C\cap S(0,1))} \leq C \sqrt{\varepsilon}
  \norm{G}_{H^1(\mathcal C\cap B(0,1))}
  \]
  by the trace-theorem, and $C$ is independent of $\varepsilon$.
  However $dx = \varepsilon^2 dy$, and $\nabla_y G(y) = \nabla_y
  (g(\varepsilon y)) = \varepsilon (\nabla g) (\varepsilon y)$ so
  \[
  \norm{G}_{L^2(\mathcal C\cap B(0,1))} = \varepsilon^{-1}
  \norm{g}_{L^2(\mathcal C\cap B(0,\varepsilon))}, \quad \norm{\nabla
    G}_{L^2(\mathcal C\cap B(0,1))} = \norm{\nabla g}_{L^2(\mathcal
    C\cap B(0,\varepsilon))}
  \]
  in other words $\norm{G}_{H^1(\mathcal C\cap B(0,1))} \leq
  \varepsilon^{-1} \norm{g}_{H^1(\mathcal C\cap
    B(0,\varepsilon))}$. This implies
  \[
  \norm{g}_{L^2(\mathcal C\cap S(0,\varepsilon))} \leq C
  \varepsilon^{-1/2} \norm{g}_{H^1(\mathcal C\cap B(0,\varepsilon))}
  \leq C \varepsilon^{-1/2} \norm{u-u'}_{H^2(\mathcal C\cap
    B(0,\varepsilon))}
  \]
  where $C$ is independent of $\varepsilon$. Combining these with
  \eqref{sphereIntegral} gives
  \[
  \ldots \leq C \sqrt{\theta_M-\theta_m} \norm{u-u'}_{H^2(\mathcal
    C\cap B(0,\varepsilon))}
  \]
  which tends to zero when $\varepsilon\to0$ because
  $\norm{u-u'}_{H^2(\mathcal C\cap B)}$ is finite.
\end{proof}

\begin{proposition} \label{wideCornerSourceValue}
  Let $\Omega\subset\R^2$ be a bounded domain and
  $-\pi<\theta_m<\theta_M<\pi$ with $\theta_M\neq\theta_m+\pi$. Assume
  that $0\in\partial\Omega$ is the centre of a ball $B$ for which
  $\Omega\cap B=\mathcal C\cap B$ where
  \[
  \mathcal C = \{x\in\R^2 \mid x\neq0,\, \theta_m < \arg(x_1+ix_2) <
  \theta_M \}
  \]
  is a cone of opening angle $\theta_M-\theta_m$ and vertex $0$.

  Let $u,u'\in H^2(\Omega\cap B)$ and $f,f'\in
  C^\alpha(\overline{\Omega\cap B})$, $\alpha>0$ solve
  \[
  \Delta u = f, \qquad \Delta u' = f'
  \]
  in $\Omega\cap B$. If $u=u'$ and $\partial_\nu u = \partial_\nu u'$
  on $\partial\Omega\cap B$ then
  \[
  f(0) = f'(0).
  \]
\end{proposition}

\begin{proof}
  \Cref{u0intByParts} gives
  \[
  \int_{\mathcal C \cap B} (f-f') u_0(sx) dx = \int_{\mathcal C \cap
    \partial B} \big( u_0(sx) \partial_\nu (u-u') - (u-u')
  \partial_\nu (u_0(sx)) \big) d\sigma
  \]
  when $s>0$ and $u_0$ is given by \cref{decaySector}.

  Let us split and estimate the integral over $\mathcal C\cap B$. By
  the H\"older-continuity of $f$ and $f'$ we get the splitting
  \begin{align}
    f(x) &= f(0) + \delta f(x),& \abs{\delta f(x)} &\leq
    \norm{f}_{C^\alpha} \abs{x}^\alpha,\label{fEstim}\\ f'(x) &= f'(0)
    + \delta f'(x), & \abs{\delta f'(x)} &\leq \norm{f'}_{C^\alpha}
    \abs{x}^\alpha.\label{f'Estim}
  \end{align}
  By \cref{2DupperAndLower} the map $x\mapsto u_0(sx)$ decays well
  enough for the following telescope identity to be well-posed
  \begin{align*}
    &\int_{\mathcal C \cap \partial B} \big( u_0(sx) \partial_\nu
    (u-u') - (u-u') \partial_\nu (u_0(sx)) \big) d\sigma \\ &\qquad =
    \int_{\mathcal C\cap B} (f-f') u_0(sx) dx = (f(0)-f'(0))
    \int_{\mathcal C\cap B} u_0(sx) dx \\ &\qquad\quad +
    \int_{\mathcal C\cap B} \delta f(x) u_0(sx) dx - \int_{\mathcal
      C\cap B} \delta f'(x) u_0(sx) dx,\\ &\int_{\mathcal C \cap B}
    u_0(sx) dx = \int_{\mathcal C \setminus B} u_0(sx) dx -
    \int_{\mathcal C} u_0(sx) dx,
  \end{align*}
  and hence
  \begin{align}
    &(f(0)-f'(0)) \int_{\mathcal C} u_0(sx) dx = (f(0)-f'(0))
    \int_{\mathcal C \setminus B} u_0(sx) dx \notag\\ &\qquad \quad+
    \int_{\mathcal C\cap B} \delta f(x) u_0(sx) dx - \int_{\mathcal
      C\cap B} \delta f'(x) u_0(sx) dx \notag\\ &\qquad \quad -
    \int_{\mathcal C \cap \partial B} \big( u_0(sx) \partial_\nu
    (u-u') - (u-u') \partial_\nu (u_0(sx)) \big)
    d\sigma(x). \label{2DsourceSplitting}
  \end{align}

  We will estimate the various terms in the identity
  \eqref{2DsourceSplitting} next. The first term on the right-hand
  side decays exponentially as $s\to\infty$ because of
  \cref{2DupperAndLower}. For the next two terms use \eqref{fEstim}
  and \eqref{f'Estim} and then the same proposition again. This gives
  the bound
  \[
  \int_{\mathcal C} \abs{\delta f(x)} \abs{u_0(sx)} dx \leq \frac{2
    (\theta_M-\theta_m) \Gamma(2\alpha+4)}{\delta_{\mathcal
      C}^{2\alpha+4}} \norm{f}_{C^\alpha} s^{-\alpha-2}
  \]
  and similarly for $\delta f'$. The boundary integral's absolute
  value can be estimated by first noting that if $r$ is the radius of
  $B$ and $x\in\partial B$, we have
  \[
  \abs{u_0(sx)} = e^{\sqrt{s r} \cos(\theta/2+\pi)} \leq
  e^{-\delta_{\mathcal C}\sqrt{r} \sqrt{s}}
  \]
  and
  \begin{align*}
    &\abs{\partial_\nu (u_0(sx))} = \abs{ \frac{\sqrt{s}
        e^{i(\theta/2+\pi)}}{2\sqrt{r}} e^{\sqrt{sr}
        \exp(i(\theta/2+\pi))}} = \frac{1}{2} \sqrt{\frac{s}{r}}
    e^{\sqrt{sr}\cos(\theta/2+\pi)} \\ &\qquad\leq \frac{1}{2}
    \sqrt{\frac{s}{r}} e^{-\delta_{\mathcal C}\sqrt{r}\sqrt{s}}
  \end{align*}
  both of which decay exponentially as $s\to\infty$. This implies the
  same for their $L^2(\mathcal C \cap \partial B)$-norm. Note that the
  same norm of $u-u'$ and $\partial_\nu(u-u')$ can be estimated above
  by $\norm{u-u'}_{H^2(\mathcal C \cap B)} \leq
  \norm{u-u'}_{H^2(\Omega)}$ according to the trace theorem. Hence
  \[
  \abs{\int_{\mathcal C \cap \partial B} \big( u_0(sx) \partial_\nu
    (u-u') - (u-u') \partial_\nu (u_0(sx)) \big) d\sigma} \leq C
  e^{-c' \sqrt{s}}
  \]
  when $s\to\infty$ for some $c'>0$.

  For the lower bound of the left of \eqref{2DsourceSplitting} use the
  identity in \cref{2DupperAndLower}
  \[
  \abs{ (f(0)-f'(0)) \int_{\mathcal C} e^{\rho\cdot x} dx } =
  \abs{f(0)-f'(0)} 6 \abs{1 - e^{2(\theta_M-\theta_m)i}} s^{-2}.
  \]
  Since $\theta_M-\theta_m\not\in \pi\Z$ the above does not vanish
  unless $f(0)=f'(0)$. Letting $s\to\infty$ in
  \eqref{2DsourceSplitting} and combining it with the estimates of the
  previous paragraphs implies that $f(0)=f'(0)$.
\end{proof}

We prove that the above proposition can be used also in three
dimensions next. This is done by a dimension reduction argument
involving the Fourier transform along the edge of a polyhedron. One
could also add a smooth change of coordinates to deal with curvilinear
edges, however we shall skip that in this article to keep the argument
simple.
\begin{lemma}\label{dimensionReduction}
  Let $\mathcal C\subset\R^{n-1}$ be a locally Lipschitz set, $L>0$,
  $\alpha>0$ and $u,u'\in H^2(\mathcal C \times {]{{-L},L}[}) \cap
  C^\alpha(\overline{\mathcal C} \times {[{{-L},L}]})$ and $f,f' \in
  C^\alpha(\overline{\mathcal C} \times {[{{-L},L}]})$. Write
  $x=(x',x_n)$ and assume that
  \begin{align*}
    &\Delta u(x) = f(x), \quad \Delta u'(x) = f(x), && x' \in \mathcal
    C, {-L}<x_n<L\\ &u=u', \quad \partial_\nu u = \partial_\nu u', &&
    x'\in\Gamma, {-L}<x_n<L
  \end{align*}
  for some open $\Gamma\subset\partial\mathcal C$. Let $\varphi\in
  C^\infty_0({]{{-L},L}[})$ and fix $\xi\in\R$. Define the dimension
  reduction operator $T_\xi$ by
  \[
  T_\xi g(x') = \int_{-L}^L e^{-i x_n \xi} \varphi(x_n) g(x',x_n) dx_n
  \]
  where $x'\in\mathcal C$.

  Then $T_\xi u, T_\xi u'\in H^2(\mathcal C) \cap C^\alpha(\mathcal
  C)$ and there is $F_\xi, F'_\xi \in C^\alpha(\mathcal C)$ such that
  \begin{align*}
    &\Delta T_\xi u (x') = F(x'), \quad \Delta T_\xi u'(x') = F'(x'),
    && x' \in \mathcal C\\ &T_\xi u = T_\xi u', \quad \partial_\nu
    T_\xi u = \partial_\nu T_\xi u', && x'\in\Gamma.
  \end{align*}
  Finally
  \[
  (F_\xi-F'_\xi)(x') = T_\xi(f-f')(x'), \qquad x'\in\partial\Gamma.
  \]
\end{lemma}
\begin{proof}
  Let us start by showing that $T_\xi:H^2(\mathcal
  C\times{]{{-L},L}[}) \to H^2(\mathcal C)$. Let $u\in
  C^\infty(\overline{\mathcal C} \times {[{{-L},L}]})$. Dominated
  convergence implies $\partial_{x'}^\beta T_\xi u (x') = T_\xi
  \partial_{x'}^\beta u (x')$ for any multi-index $\beta\in\N^{n-1}$,
  so
  \[
  \abs{ \partial_{x'}^\beta T_\xi u(x') } \leq \int_{-L}^L
  \norm{\varphi}_\infty \abs{\partial_{x'}^\beta u(x',x_n)} dx_n,
  \]
  which gives
  \[
  \norm{T_\xi u}_{H^2(\mathcal C)} \leq \norm{\varphi}_\infty
  \norm{u}_{H^2(\mathcal C \times {]{{-L},L}[})}
  \]
  by the Minkowski integral inequality. This gives a unique bounded
  extension to $u\in H^2$. The case of $u\in
  C^\alpha(\overline{\mathcal C} \times {[{{-L},L}]})$ follows even
  more directly:
  \[
  \abs{ T_\xi u(x') - T_\xi u(y') } \leq 2\norm{\varphi}_\infty
  \norm{u}_{C^\alpha} \abs{x'-y'}^\alpha.
  \]

  For the identities concerning $\Delta T_\xi u$, recall that
  $\Delta_{x'} u = f - \partial_{x_n}^2 u$. Integration by parts gives
  \begin{align*}
    &\Delta_{x'} T_\xi u(x') = T_\xi f(x') - T_\xi \partial_{x_n}^2 u
    (x') = - \int_{-L}^L e^{-ix_n\xi_n} \varphi''(x_n) u(x',x_n) dx_n
    \\ &\qquad + 2i\xi_n \int_{-L}^L e^{-ix_n\xi_n} \varphi'(x_n)
    u(x',x_n) dx_n + \xi_n^2 T_\xi u(x') + T_\xi f(x')
  \end{align*}
  and we let $F_\xi(x')$ be the right-hand side above. Do the same for
  $u'$ and $f'$. These are easily seen to be in $C^\alpha$ with
  respect to $x'$, just as in the previous paragraph.

  Since $u=u'$ when $x'\in\Gamma$, the same holds for $T_\xi u$ and
  $T_\xi u'$.  Also, $\partial_\nu u = \partial_\nu u'$ on
  $x'\in\Gamma$ implies the same for $\partial_\nu T_\xi u$ and
  $\partial_\nu T_\xi u'$, because $\nu \cdot e_n = 0$ in
  $\R^n$. These imply a fortiori that $F_\xi(x') - F'_\xi(x') =
  T_\xi(f-f')(x')$ for $x'\in\Gamma$.
\end{proof}

\begin{proposition} \label{wideEdgeSourceValue}
  Let $\Omega\subset\R^3$ be a bounded domain with
  $0\in\partial\Omega$. Let $-\pi<\theta_m<\theta_M<\pi$ with
  $\theta_M\neq\theta_m+\pi$ and define
  \[
  \mathcal C = \{x\in\R^2 \mid x\neq0,\, \theta_m < \arg(x_1+ix_2) <
  \theta_M \}.
  \]
  Let $B\subset\R^2$ be an origin-centred ball and assume that there
  is $L>0$ such that
  \[
  \big(B\times{]{{-L},L}[}\big) \cap \Omega = (B\cap\mathcal C) \times
      {]{{-L},L}[}
  \]
  i.e. $\Omega$ has an edge of opening angle $\theta_M-\theta_m$.

  Let $u,u'\in H^2((B\times{]{{-L},L}[}) \cap \Omega)$ and $f,f'\in
  C^\alpha(\overline{(B\times{]{{-L},L}[}) \cap \Omega})$, $\alpha>0$
  solve
  \[
  \Delta u = f, \qquad \Delta u' = f'
  \]
  in $B\cap\Omega$. If $u=u'$ and $\partial_\nu u = \partial_\nu u'$
  on $B\cap\partial\Omega$ then
  \[
  f(0) = f'(0).
  \]
\end{proposition}
\begin{proof}
  $H^2$ embeds into $C^{1/2}$ by the Sobolev embedding, and we may
  assume that $\alpha\leq 1/2$. For any $\xi\in\R$
  \cref{dimensionReduction} implies the existence of $U,U'\in
  H^2(B\cap\mathcal C) \cap C^\alpha(\overline{B\cap\mathcal C})$ and
  $F_\xi,F'_\xi \in C^\alpha(\overline{B\cap\mathcal C})$ such that
  \[
  \Delta U = F_\xi, \qquad \Delta U' = F'_\xi
  \]
  in $B\cap\mathcal C$ and on $B\cap\partial\mathcal C$ we have
  $U=U'$, $\partial_\nu U = \partial_\nu
  U'$. \Cref{wideCornerSourceValue} implies that $F_\xi(0) =
  F'_\xi(0)$. Since
  \[
  0 = (F_\xi-F'_\xi)(0) = \int_{-L}^L e^{-ix_n\xi} \varphi(x_n)
  (f-f')(0,x_n) dx_n
  \]
  for any given $\varphi\in C^\infty_0({]{{-L},L}[})$ and for all
  $\xi\in\R$, the Fourier inversion formula shows that $f(0) = f'(0)$.
\end{proof}

\section{Theorems}
\begin{theorem}[Sources with corners or edges radiate]\label{thm1}
  Let $f = \chi_\Omega \varphi$ for a bounded domain
  $\Omega\subset\R^n$, $n\in\{2,3\}$ and bounded function $\varphi \in
  L^\infty(\R^n)$.

  Assume that $\Omega$ has a corner (2D) or an edge (3D) and that
  $\varphi$ is H\"older-continuous near it. Moreover there must be a
  chain of balls in $\R^n\setminus\Omega$ connecting this corner or
  edge to infinity.

  Let $k>0$ and $u\in H^2_{loc}(\R^n)$ have acoustic source $f$,
  namely
  \[
  (\Delta+k^2) u = f, \qquad \lim_{\abs{x}\to\infty}
  \abs{x}^{\frac{n-1}{2}} \left(\frac{x}{\abs{x}} \cdot \nabla -
  ik\right) u(x) = 0
  \]
  where the limit is uniform over the direction $x/\abs{x}$. If $u(x)$
  has zero far-field pattern, then $\varphi=0$ on the corner or edge
  i.e. $f$ has no jumps at these locations.
\end{theorem}
\begin{proof}
  Rellich's theorem and unique continuation imply that $u=0$ in the
  connected component of $\R^n\setminus\Omega$ that reaches
  infinity. Hence $u=0$ and $\partial_\nu u=0$ on the boundary of the
  corner or edge. On the other hand
  \[
  \Delta u = f-k^2 u, \qquad \Delta 0 = 0
  \]
  in $\R^n$ and all the smoothness assumptions of
  \cref{wideCornerSourceValue} and \cref{wideEdgeSourceValue} are
  satisfied. They imply that $\varphi-k^2u = 0$ at the corner or
  edge. However $u$ is zero there, so $\varphi$ vanishes too.
\end{proof}

In \cite{BLvanishing} we showed that under the specialised condition
below and certain geometric assumptions the transmission
eigenfunctions vanish at every corner of $\Omega$.
\begin{quote}
  ``If $w$ can be approximated in the $L^2(\Omega)$-norm by a sequence
  of Herglotz waves with uniformly $L^2(\mathbb S^{n-1})$-bounded
  kernels, then
  \[
  \lim_{r\to0} \frac{1}{m(B(x_c,r))} \int_{B(x_c,r)} \abs{w(x)} dx = 0
  \]
  where $x_c$ is any vertex of $\Omega$ such that $V(x_c) \neq 0$.''
\end{quote}
Using the new techniques we show the following
\begin{theorem} \label{vanishingThm}
  Let $n\in\{2,3\}$ and $\Omega\subset\R^n$ be a bounded domain. Let
  $V\in L^\infty(\Omega)$. Assume that $k>0$ is a transmission
  eigenvalue: there exists $v,w \in L^2(\Omega)$ such that \eqref{TE}
  holds.

  \smallskip
  Let $x_c$ be any vertex or edge of $\Omega$ such that $V$ is
  $C^\alpha$ smooth, $\alpha>0$ near $x_c$.  If $v$ or $w$ is
  $H^2$-smooth in a neighbourhood of $x_c$ in $\Omega$ then
  $w(x_c)=v(x_c)=0$ if $V(x_c) \neq 0$.
\end{theorem}
This new assumption is simply a condition on the boundary: elliptic
regularity guarantees $H^2$-smoothness in any domain in $\Omega$ whose
boundary is disjoint from $\partial\Omega$. This is generally true
based on numerical evidence \cite{BLvanishingNum}. That paper also
shows a case where the transmission eigenfunctions are not
$H^2$-smooth: if the corner is not convex then actually $v$ and $u$
blow up. This observation screams for a mathematical proof.

\begin{proof}[Proof of \cref{vanishingThm}]
  Set $f=k^2(1+V)v$, $f'=k^2w$ and $u=v$, $u'=w$. Let $B\subset\R^n$
  be an $x_c$-centred ball of sufficiently small radius so that $V\in
  C^\alpha(\overline{B\cap\Omega})$ and so that $v,w\in
  H^2(B\cap\Omega)$. In two and three dimensions this embeds into
  $C^{1/2}$, and we may assume that $\alpha\leq1/2$.

  Let $x_c\in\partial\Omega$ be a corner (2D) or edge point (3D). The
  source terms $f,f'$ are H\"older-continuous and $u,u'$ are
  $H^2$-smooth in $B\cap\Omega$. Also $u=u'$ and $\partial_\nu
  u=\partial_\nu u'$ on $\partial\Omega$ near $x_c$ because
  $u-u'=v-w\in H^2_0(\Omega)$. \Cref{wideCornerSourceValue} and
  \cref{wideEdgeSourceValue} imply that $f(x_c)=f'(x_c)$. However
  $v(x_c)=w(x_c)$ so the latter implies $k^2V(x_c)v(x_c)=0$. The
  conclusion follows.
\end{proof}

Finally we can also determine the shape of a convex polyhedral source,
and also the source values at corners and edges. The two-dimensional
case was shown in \cite{Ikehata} earlier.

\begin{theorem}[Source shape and value determination]\label{thm2}
  Let $n\in\{2,3\}$ and $\Omega,\Omega'\subset\R^n$ be bounded convex
  polyhedral domains. Let $\varphi,\varphi' \in L^\infty(\R^n)$ be
  $C^\alpha$-smooth, $\alpha>0$ in some neighbourhoods of the vertices
  of $\Omega,\Omega'$, respectively, and have nonzero value there.

  Define $f = \chi_\Omega\varphi$, $f' = \chi_{\Omega'}\varphi'$. Let
  $k>0$ and $u,u'\in H^2_{loc}(\R^n)$ have acoustic sources $f,f'$,
  namely
  \[
  (\Delta+k^2) u = f, \qquad (\Delta+k^2) u' = f',
  \]
  and
  \[
  \lim_{\abs{x}\to\infty} \abs{x}^{\frac{n-1}{2}}
  \left(\frac{x}{\abs{x}} \cdot \nabla - ik\right) u(x) =
  \lim_{\abs{x}\to\infty} \abs{x}^{\frac{n-1}{2}}
  \left(\frac{x}{\abs{x}} \cdot \nabla - ik\right) u'(x) = 0
  \]
  where the limit is uniform over the direction $x/\abs{x}$.

  If $u$ and $u'$ have the same far-field pattern then
  $\Omega=\Omega'$ and $\varphi=\varphi'$ at each of their
  vertices. In three dimensions, if $\varphi$ and $\varphi'$ are
  H\"older-continuous near the edges, then also $\varphi=\varphi'$ on
  these edges.
\end{theorem}

\begin{proof}
  By Rellich's lemma and unique continuation (e.g. Lemma 2.11 in
  \cite{CK}) we see that $u=u'$ in $\R^n\setminus(\Omega\cup\Omega')$.

  \smallskip
  We prove first that $\Omega=\Omega'$ by showing that the opposite
  leads to a contradiction. Assume $\Omega\not\subset\Omega'$. Then by
  convexity there is a corner (2D) or edge (3D) point
  $x_c\in\partial\Omega\setminus\overline{\Omega'}$ such that
  $\varphi\in C^\alpha$ near it and $\varphi(x_c)\neq0$. Since $u=u'$
  outside $\Omega\cup\Omega'$ we have $u=u'$ and $\partial_\nu
  u=\partial_\nu u'$ on $\partial\Omega$ near $x_c$. We also have
  \[
  \Delta u = \varphi-k^2u, \qquad \Delta u' = -k^2u'
  \]
  in $\Omega$ near $x_c$. Also, $u,u'\in H^2$ there, and it embeds
  into $C^{1/2}$ in two and three dimensions. Hence the source terms
  $\varphi-k^2u$ and $-k^2u'$ are H\"older-continuous.
  \Cref{wideCornerSourceValue} and \cref{wideEdgeSourceValue}
  imply that $\varphi(x_c)-k^2u(x_c)=-k^2u'(x_c)$. But we already know
  that $u(x_c)=u'(x_c)$. Thus $\varphi(x_c)=0$ but this is a
  contradiction with the choice of $x_c$. Hence
  $\Omega\subset\Omega'$. Similarly $\Omega'\subset\Omega$.

  \smallskip
  Now, after $\Omega=\Omega'$ let $x_c$ be any of its vertices in 2D
  or an edge point close to a vertex in 3D. We note that
  \[
  \Delta u = \varphi-k^2u, \qquad \Delta u' = \varphi'-k^2u'
  \]
  in $\Omega$ and $u,u'\in H^2$ which embeds into $C^{1/2}$. So the
  right-hand sides above are H\"older-continuous. Taking into account
  that $u=u'$ and $\partial_\nu u = \partial_\nu u'$ on
  $\partial\Omega$ near $x_c$, \cref{wideCornerSourceValue} and
  \cref{wideEdgeSourceValue} imply $\varphi-k^2u=\varphi'-k^2u'$ at
  $x_c$, and so $\varphi(x_c)=\varphi'(x_c)$. The same deduction works
  if $x_c$ is any arbitrary edge point near which $\varphi,\varphi'$
  are H\"older-continuous.
\end{proof}

\section{Relation to the enclosure method}\label{IkehataComparison}
After seeing a first version of this manuscript Professor Masaru
Ikehata kindly pointed us in the direction of \cite{Ikehata} and other
one measurement work on the enclosure method \cite{IkehataSurvey,
  IkehataConductivity, Friedman--Isakov, IkehataScattering,
  IkehataObstacle}. In the former he shows the following two results
\begin{theorem*}
  Let $u\in H^1(\Omega)$ with $\Omega\subset\R^2$ a bounded Lipschitz
  domain. Let $F\in (H^1_0(\Omega))^\ast$ and assume
  \[
  (\Delta+k^2)u = F, \qquad \Omega
  \]
  for a fixed $k>0$. If $F$ is compactly supported in $\Omega$ with
  $F=\chi_D\varphi$, $D$ a polygon away from $\partial\Omega$ and
  $\varphi\in L^\infty$ H\"older-continuous and non-vanishing near the
  vertices of $D$, then the convex hull of $D$ can be calculated from
  the knowledge of $(u,\partial_\nu u)$ on $\partial\Omega$.
\end{theorem*}

\begin{theorem*}
  Let $u\in H^2(\Omega)$ with $\Omega\subset\R^2$ a bounded Lipschitz
  domain. Let $V=\chi_D\varphi$ with $D\subset\Omega$ a polygon away
  from $\partial\Omega$, and $\varphi\in L^\infty(\Omega)$
  H\"older-continuous and non-vanishing near any vertex of $D$. Let $u$
  satisfy
  \[
  (\Delta+k^2(1+V))u = 0, \qquad \Omega
  \]
  for a fixed $k>0$. Then $(u,\partial_\nu u)$ on $\partial\Omega$
  determines the convex hull of $D$ assuming that $u\neq0$ on any
  vertex of $D$.
\end{theorem*}

These two theorems show that the enclosure method can be applied in
the context of a single measurement and polygonal geometry. The first
theorem above is equivalent to \cref{thm2} modulo smoothness index and
in two dimensions. Its proof starts with the integral identity
\[
\int_D \varphi u_0 dx = \int_{\partial\Omega} \big( u_0 \partial_\nu u
- u \partial_\nu u_0 \big) dx,
\]
with the choice of using complex planar waves of the form
\[
u_0(x) = e^{x\cdot(\tau\omega+i\sqrt{\tau^2+k^2} \omega^\perp) - \tau
  t}
\]
where $\omega,\omega^\perp\in\mathbb S^1$,
$\omega\cdot\omega^\perp=0$, and he then calculates the exact
asymptotic behaviour of the integral $\int_D \varphi u_0 dx$ as
$\tau\to\infty$ but $t\in\R$ is fixed. It turns out that the principal
term arises, depending on $\omega$, on a corner of $D$, and is
exponentially growing or decaying depending on whether $t$ is too
large to too small. This gives an indicator function for the convex
hull of $D$.

The argument of \cite{Ikehata} looks quite similar to
\cref{sourceIdentity} and the proof of \cref{wideCornerSourceValue},
if not for the use of the different exponential solutions of
\cref{2Du0Def}: both use a similar integral identity and calculate the
asymptotic expansion of an integral while noting that the main
contribution comes from the value at a corner. In this sense Ikehata's
enclosure method and the implied reconstruction algorithms should be
of great interest to people studying corner scattering and the
transmission eigenvalue problem at corners.

The above is a very general idea, and one which was the basis of
studying corner scattering starting from \cite{BPS}. After a deeper
study of Professor Ikehata's article \cite{Ikehata} we can claim the
following. The enclosure method with a single measurement of
\cite{Ikehata} applies to
\begin{itemize}
\item inverse source problems to recover the convex hull of a
  polyhedral source, or show that it must vanish at corners, or
\item in inverse scattering if the total wave does not vanish, to
  recover the convex hull of a polyhedral scatterer.
\end{itemize}
This implies the following. Consider our recent articles
\cite{BLstability, BLpiecewise, BLvanishing}. If one could modify
their proofs to use \cref{sourceIdentity} instead of the previously
used identity
\[
k^2 \int_\Omega (V-V') u' u_0 dx = \int_{\partial\Omega} \big( u_0
\partial_\nu (u-u') - (u-u') \partial_\nu u_0 \big) d\sigma
\]
with $(\Delta+k^2(1+V))u=0$, $(\Delta+k^2(1+V'))u'=0$,
$(\Delta+k^2(1+V))u_0=0$, then the enclosure method could give some
new insight. In particular it would give hope for reconstruction
formulas. This merits further study and collaboration.

Coming back to the results of this paper, we note that if we required
all corners and edges to be convex in \cref{thm1,vanishingThm}, then
\cref{thm2} and a-fortiori in two dimensions \cite{Ikehata} would
imply them. In essence the enclosure method with one measurement works
well for determining the shape in the source problem. However it seems
to have difficulties extracting information about non-convex corners
of the source or obstacle from a single far-field or Cauchy data
measurement.

\section{Conclusions}
In this article we studied the inverse source problem with acoustic
sources having a corner or edge. We showed that nonradiating sources
must vanish at convex and non-convex corners or edges in two and three
dimensions. This is in stark contrast with one dimension and also the
existence of $n$-dimensional spherical constant sources that are known
to be nonradiating. We also showed that a far-field pattern determines
the shape and corner and edge values of acoustic sources in the class
of polyhedral sources.

The above results were made possible by an integration by parts
formula that uses harmonic functions as test functions instead of
complex geometrical optics solutions. A new type of harmonic
exponential solution that decays in almost all directions allowed
considering non-convex corners and edges. Finally, these tools gave a
simple proof for the vanishing of $H^2$-smooth transmission
eigenfunctions on corners and edges.

Some interesting questions still remain open. Our proof has similar
major ideas as Ikehata's enclosure method. Could this give new insight
for example for reconstruction? Considering nonradiation, the inverse
source problem for the Maxwell equations is of great interest because
it has a richer set of nonradiating sources than the Helmholtz
equation \cite{Gbur}. In addition there is still the open question of
under what geometrical conditions are the transmission eigenfunctions
$H^2$-smooth.

\bibliography{source}{} \bibliographystyle{plainnat}

\end{document}